\documentclass{article} 
\usepackage[english]{babel}
\usepackage{amsmath,amssymb,bbm,amsthm}
\usepackage[matrix,arrow,curve]{xy}

\usepackage{graphicx} 
\usepackage{pstricks} 
\usepackage{psfrag}

\usepackage{eepic} 
\newcommand{\assign}{:=}

 \newcommand{\tmem}[1]{{\em #1\/}}
\newcommand{\tmop}[1]{\ensuremath{\operatorname{#1}}}
\newcommand{\tmstrong}[1]{\textbf{#1}} \newcommand{\si}{{SI}}
\newcommand{\wh}{\tmop{Wh}} \newcommand{\db}{{D}}
\renewcommand{\ss}{\mathcal S^s} \newcommand{\sh}{\mathcal S^h}
\renewcommand\stop{\mathcal S}

\newtheorem{dummy}{anything}[section]
\newtheorem{theorem}[dummy]{Theorem} 
\newtheorem{lemma}[dummy]{Lemma}
\newtheorem{prop}[dummy]{Proposition} 
\newtheorem{cor}[dummy]{Corollary}
\newtheorem{defn}[dummy]{Definition} 
\newtheorem{rems}[dummy]{Remarks}
\newtheorem{rem}[dummy]{Remark} 
\newtheorem{question}[dummy]{Question}
\newcommand{\dia}[1]{\begin{array}{c}{\xymatrix@C-3pt@M+2pt@R-4pt{#1
}}\end{array}} \newcommand{\bbz}{\mathbb Z}
\newcommand{\calh}{\ensuremath{\mathcal H}}

\begin{document}

\title{Whitehead torsion of inertial $h$-cobordisms\footnote{2010
Mathematics Subject Classification 57R80(primary), 57R67(secondary).
The second author is supported by the Simons Foundation Grant 281810.
The first author would like to thank Tulane University for its 
hospitality during part of this research.} 
}

\author{Bj\o rn Jahren and S\l awomir Kwasik}

\date{}
\maketitle

\begin{abstract}
\noindent We study the Whitehead torsions of  inertial
$h$-cobordisms, and identify various types representing a nested 
sequence of subsets of the Whitehead group.  A number of examples are 
given to show that these subsets are all different in general.
\end{abstract}

\section{Introduction}

The $h$-cobordism theorem plays a crucial role in modern geometric
topology, providing the essential link between homotopy and geometry.
Indeed, comparing manifolds of the same homotopy type, one can often use
surgery methods to produce $h$-cobordisms between them, and then hope to
be able to show that the Whitehead torsion $\tau ({W}^{n +1}, {M}^n)$ in
$\wh (\pi_1 ({M}^n))$ is trivial. By the $s$-cobordism theorem, the two
manifolds will then be isomorphic (homeomorphic or diffepmorphic,
according to  which category we work in).

The last step, however, is in general very difficult, and what makes the
problem even more complicated, but at the same time more interesting, is
that there exist $h$-cobordisms with non-zero torsion, but were the ends
still are isomorphic (cf. \cite{Hs1}, \cite{Hs2}, \cite{L},
\cite{JK}). Such $h$-cobordisms we call {\em inertial}. The central
problem is then to determine the subsets of elements of the Whitehead
group $\wh (\pi_1 ({M}^n))$ which can be realized as Whitehead torsion
of inertial $h$-cobordisms. This is in general very difficult, and only
partial results in this direction are known  (\cite{Hs1}, \cite{Hs2},
\cite{L}).  

The purpose of this note is to shed some light on this important
problem.  \

\section{Inertial $h$-cobordisms}

In this section we recall basic notions and constructions concerning
various types of $h$-cobordisms. We will follow the notation and
terminology of \cite{JK}. For convenience we choose to formulate
everything in the category of topological manifolds, but for most of
what we are going to say, this does  not make much difference. See
Section \ref{TopInv} for more on the relations between the different
categories.\smallskip

An $h$-cobordism $({W}; {M}, {M}')$ is a manifold ${W}$ with two
boundary components ${M}$ and ${M}'$, each of which is a deformation
retract of ${W}$.

We will think of this as an $h$-cobordism from ${M}$ to ${M}'$, thus
distinguishing it from the dual $h$-cobordism $({W}; {M}', {M})$. Since
the pair $({W}; {M})$ determines ${M}'$, we will often use the notation
$({W}; {M})$ for $({W}; {M}, {M}')$.  We denote by $\calh(M)$ the set of
homeomorphism classes relative ${M}$ of $h$-cobordisms from ${M}$.

If $X$ is  a path connected space, we denote by $\wh (X)$  the Whitehead
group $\wh (\pi_1 (X))$. Note that this is independent of choice of base
point of $X$, up to unique isomorphism.

The $s$-cobordism theorem (cf. \cite{Ma}, \cite{Mi}) says that if ${M}$
is compact connected (closed) and of dimension at least 5 there is a
one-to-one correspondence between $\calh({M})$ and $\wh ({M})$
associating to the $h$-cobordism $({W}; {M}, {M}')$ its Whitehead
torsion $\tau ({W}; {M}) \in \wh ({M})$. Given an element $({W}; {M},
{M}') \in \calh ({M})$ the restriction of a retraction $r:W\to M$ to
$M'$ is a  homotopy equivalence $h:M'\to M$, uniquely determined up to
homotopy.  By a  slight abuse of language, any  such $h$ will be
referred to as  ``the natural homotopy equivalence''.   It induces a
unique isomorphism 
\[ h_{\ast} : \wh ({M}') \rightarrow \wh ({M}). 
\]

Recall also that there is an involution $\tau \rightarrow \bar{\tau}$ on
$\wh ({M})$ induced by tranposition of matrices and inversion of group
elements (cf. \cite{Mi}, \cite{O}). If $M$ is non-orientable, the
involution is also twisted by the orientation character $\omega:
\pi_1(M)\to \{\pm 1\}$, i.\,e. inversion of group elements is replaced
by $\tau\mapsto \omega(\tau)\tau^{-1}$.

Let $({W}; {M}, {M}')$ and $({W};{M}', {M})$ be dual  $h$-cobordisms
with  ${M},\!{M}'$ of dimension $n$.  Then $\tau(W;M)$ and $\tau(W;M')$
are related by the  basic duality formula (cf. \cite{Mi}, \cite{JK})
\[ h_{\ast} (\tau ({W}; {M}')) = (- 1)^n \overline{\tau({W}; {M})}. \]

We refer to Section \ref{comments} for further discussion of Whitehead
torsion.\smallskip

\begin{defn} The {\em{inertial set}} of a manifold ${M}$ is defined as
\[ I ({M}) = \{ ({W}; {M}, {M}') \in\calh({M}) |{M}\cong{M}' \}, \] or
the corresponding subset of $\wh ({M})$.
\end{defn}

There are many ways to construct inertial $h$-cobordisms. Here we recall
three of these.

{\tmstrong{A}}. Let $G$ be an arbitrary (finitely presented) group. Then
there is a 2-dimensional simplicial complex $K$ (finite) with $\pi_1 (K)
\cong G$.  Let $\tau_0 \in \wh (G)$ be an element with the property that
$\tau_0 = \tau (f)$ for some homotopy self-equivalence $f : K
\rightarrow K$. Denote by $N (K)$ a regular neighborhood of $K$ in a
high-dimensional Euclidean space $\mathbbm{R}^n$($n \geqslant 5$ will
do). Approximate $f : K \rightarrow K \subseteq N (K)$ by an embedding
whose image has neighborhood  $N'(K)\subset \mathrm{int}\,N(K)$. By
uniqueness of neighborhoods, $N'(K)\approx N(K)$.  Then ${W}= N (K)
-\mathrm{int}\, N'(K))$ is an inertial $h$-cobordism whose  torsion $
\tau ({W}; \partial N' (K))$ can be identified with $\tau_0$ via the
$\pi_1$-isomorphisms $\partial N'(K)\subset N(K)\supset K$.
(cf. \cite{Hs1}, \cite{Hs2}).

\

{\tmstrong{B}}. Let $f : {M} \rightarrow {M}$ be a homotopy
self-equivalence of a closed manifold and let $\tau_0 = \tau (f) \in \wh
({M})$. Approximate $f : {M} \rightarrow {M} \subset M\times D^n$ by an
embedding (cf. \cite{Wa1}), where $D^n$ is the $n$-dimensional disk, $n$
big. In the same way as in A, this will lead to an inertial
$h$-cobordism  between two copies of ${M} \times S^{n - 1}$, with
torsion $\tau_0$ (cf. \cite{JK}).

\

{\tmstrong{C}}. Let $({W}; {M}, {M}')$ be an $h$-cobordism with  torsion
$\tau_0 = \tau ({W}; {M})$. Form the {\em double} (cf. \cite{Mi},
\cite{JK}):
\[ (\widetilde{{W}} ; {M}, {M}) \assign \left( {W} \underset{{M}'}{\cup}
{W}; {M}, {M} \right)
\] Then $\tau (\widetilde{{W}} ; {M}) = \tau_0 + (- 1)^n \bar{\tau}_0$
and this again often leads to a nontrivial inertial $h$-cobordism; for
example if $n$ is odd and the involution $- : \wh ({M}) \rightarrow \wh
({M})$ is nontrivial.\par

It will be convenient to introduce the notation $\db(M)$ for the
subgroup  $\{\tau+ (- 1)^n \bar\tau| \tau\in \wh(M)\}$ of $\wh(M)$.
Note that $\db(M)$ depends only on $\pi_1(M)$, orientation and the
dimension of $M$.

The construction in {\tmstrong{C}}  leads to $h$-cobordisms that are
particularly simple and have special properties: not only do they come
with canonical identifications of  the two ends, but they are also {\em
strongly inertial}.

\begin{defn} {\em (Cf. \cite{JK})}. The $h$-cobordism $({W}; {M}, {M'})$
is called {\em strongly inertial}, if the natural homotopy equivalence
$h : {M}' \rightarrow {M}$ is homotopic to a homeomorphism.
\end{defn}

The set of (Whitehead torsions of) strongly inertial $h$-cobordism will 
be denoted by $\si({M})$. It was observed in \cite{JK} that
$SI(M)\subseteq \wh(M)$ is a subgroup.

Obviously $\si ({M}) \subseteq I ({M})$ and there are many examples of
inertial but not strongly inertial $h$-cobordism,  for example
constructed using the methods in {\tmstrong{A}} or {\tmstrong{B}}.  In
fact, for any manifold $M$ of dimension $n \geqslant 5$, we have
$$ I(M\#_k(S^p\times S^{n-p})) = \wh (M\#_kS^p\times S^{n-p}), $$
for $2\leq p\leq n-2$ and $k$ big enough \cite{HL}.  (If $\pi_1(M)$ is
finite, $k=2$ suffices.)

However, for $\si(M)$ there are restrictions.  For example, since the
natural homopy equivalence $h$ is homotopic to a homeomorphism, its
Whitehead torsion $\tau(h)$ must vanish.  But  we have (equation
(\ref{eq:tau-h}) in Section \ref{comments})
$$
\tau(h)= -\tau(W;M)+(-1)^n\overline{\tau(W;M)},
$$
so $\tau(W;M)$ must satisfy the formula
$\tau(W;M)=(-1)^n\overline{\tau(W;M)}$,  \i.\,e. 
\begin{equation} \si(M) \subseteq A(M):=\{\tau\in \wh(M) |
\tau=(-1)^n\bar\tau\}.\label{eq:A(M)}
\end{equation}

In special cases we have even stronger restrictions, as in the following
result  (Theorem 1.3 in \cite{JK}):

\begin{theorem}\label{abelian} Suppose ${M}$ is a closed oriented
manifold of odd dimension with finite abelian fundamental group. Then
every strongly inertial $h$-cobordism from ${M}$ is trivial.
\end{theorem}

This result motivated us to look more closely at strongly invertible
$h$-cobordisms with finite fundamental groups. Our main  interest is the
following:\smallskip

\noindent{\tmstrong{Problem:}}{\tmem{ Let ${M}^n$ be a closed
$n$-dimensional (oriented) manifold with $n \geqslant 5$ and with finite
fundamental group $\pi_1 ({M}^n)$. Determine the subset $\si ({M}^n)$ of
$\wh ({M}^n)$. In particular, is $\ \si ({M}^n) =\db(M)$?}\medskip

Note that if $G$ is a finite abelian group, then the involution  $- :
\wh (G)\rightarrow \wh (G)$ is trivial (cf. \cite{O}), and consequently
$\db({M}^n) = \{ 0 \}$ for $n$ odd. Hence, in this case  $\si(M)=\db(M)$
by  Theorem \ref{abelian}.

Our first new observation is that $\si ({M}^n) = \{ 0 \}$ also for odd
dimensional manifolds ${M}^n$ with $\pi_1 ({M}^n)$ finite periodic,
namely:

\begin{theorem}\label{periodic} Let $({W}^{n + 1} ; {M}^n, {N}^n)$ be a
strongly inertial $h$-cobordism with $M$ orientable, $n$ odd and $\pi =
\pi_1 ({M}^n)$ finite periodic. Then   ${W}^{n + 1} = M^n \times I$ for
$n \geqslant 5$. Hence $\si ({M}^n) = \{ 0 \}$. 
\end{theorem}

The class of finite periodic fundamental groups has attracted a lot of
attention in topology of manifolds and transformation groups
(cf. \cite{MTW}, \cite{KS2}). The most extensive classification results
for manifolds with finite  fundamental groups involve this class of
groups.

Let ${M}^n$ be a closed, oriented manifold with $\pi_1 ({M}^n)$ finite
abelian. If $n$ is odd, then, as we observed, $\si ({M}^n) = \{ 0
\}$. In the even dimensional case the situation is quite different.

\begin{theorem}\label{DnotSI} For every $n\geqslant 3$  there are
oriented manifolds $M^{2n}$ with $\pi_1 ({M}^{2 n})$  finite cyclic and
with  $\{ 0 \}  \neq \db(M) \neq \si ({M}^{2 n})$.
\end{theorem}

The following result shows that orientability is essential in Theorem
\ref{abelian}:

\begin{theorem}\label{nonor} In every odd dimension $n\geqslant 5$ there
are closed nonorientable manifolds with finite, cyclic fundamental
groups and strongly inertial $h$-cobordisms from $M$ with nontrivial
Whitehead torsion.
\end{theorem}

Note that in this case $\db(M)$ is trivial.  \smallskip

\begin{rems} {\rm ($i$). There are obvious inclusions  $\{0\}\subset
D(M)\subset SI(M)\subset I(M) \subset \wh(M)$. In addition it is proved
in \cite{HJ} that $A(M)\subset I(M)$, such that combined with
(\ref{eq:A(M)}) we have a sequence of subsets
\begin{equation}\label{filtration} \{0\}\subset D(M)\subset SI(M)\subset
A(M)\subset I(M)\subset \wh(M).
\end{equation}

Clearly each of these inclusions can be an equality for some $M$, but
for each pair of subsets we now have examples of manifolds where the
inclusion is proper. (For $SI(M)\ne A(M)$, see e.\!g.
\cite[Example 6.4]{JK}.)\par
 
\par

($ii$)  $D(M)$ and $A(M)$ depend only on the fundamental group,  and
Khan \cite{QK}  has shown  that $\si(M)$ is homotopy invariant. It would
be interesting to know if $\si(M)$ also only depends on the fundamental
group.  If so, it is a functorial, algebraically defined subgroup of
$\wh(M)$ between $D(M)$ and $A(M)$.  What could it be? \par 
Observe also
that the quotient $A(M)/D(M)$ is equal to the Tate cohomology group
$\hat H^n(\bbz_2;\wh(M))$, where $n=\dim M$, and therefore $SI(M)/D(M)$
is a subgroup. Another description  of this subgroup is given in the
beginning of Section \ref{proofs}.\par  
  Note that Hausmann has shown
that $ I(M)$ is {\em not} homotopy  invariant, and in general is not a
subgroup of $\wh(M)$ \cite{Hs2}.  However, it is preserved by the
involution $\tau\mapsto (-1)^{n+1}\bar \tau$ \cite[Lemma 5.6]{Hs2}.}
\end{rems}  

There is one more piece of structure that we should mention: the group
$\pi_0(\tmop{Top}(M))$ of isotopy classes of homeomorphisms of $M$ acts
on $\wh(M)$ via the isomorphisms induced on the fundamental group.
(Recall that $\wh(M)$ is independent of choice of base point.)
Geometrically, this corresponds to changing an $h$-cobordism $(W;M)$ by
the way $M$ is identified with part of the boundary of $W$. Hence the
orbits represent equivalence classes under homeomorphisms preserving
boundary components, but not necessary the identity on any of them. A
simple  example to illustrate this is the case where $M=P_1\#P_2$, where
$P_1$ and $P_2$ are copies of the same manifold. Since  $\wh(M)\simeq
\wh(P_1)\oplus\wh(P_2)$ (\cite{Sta}), this means that every
$h$-cobordism from $M$ is a band-connected sum $W_1\#_{S^{n-1}\times
I}W_2$ of $h$-cobordisms from $P_1$ and $P_2$, and the homeomorphisms
interchanging $P_1$ and $P_2$ just interchanges  $W_1$ and $W_2$. 
\par

The observation now is that the action of   $\pi_0(\tmop{Top}(M))$
clearly preserves the filtration (\ref{filtration}).\par

Note that on $\wh(M)$ this action factors through an action of the group
$\pi_0(\tmop{Aut}(M))$ of homotopy classes of homotopy equivalences of
$M$.  Since the action of $\pi_0(\tmop{Aut}(M))$ is defined
algebraically,  it must also preserve the functorial subgroups $D(M)$
and $A(M)$. \par
This action  does  not have an easy geometric
interpretation, but $SI(M)$ is still preserved, by the more subtle 
functoriality of \cite[Theorem 3.1]{QK}, as explained in Corollary 
\ref{SIinv} below. 
However, it is an easy consequence of \cite[Theorem
6.1]{Hs2} that it does {\em not} in general preserve $I(M)$ .

\section{Proofs}\label{proofs}

In this section all manifolds have dimension at least five.
The proofs are based on the following commutative diagram, which is part
of the  braid (\ref{eq:rotharray}) in Section \ref{comments}.  The rows
are the  Wall-Sullivan  exact sequences for topological surgery (cf.
\cite{Wa1}, \cite{Ra}), and  the columns are part of the Rothenberg
sequences for $L$-groups and structure sets. 
$$
\dia{ L_{n + 2}^s (M) \ar[r]^{\gamma^s}\ar[d]^{l_1} & S^s({M} \times
I)\ar[r]^{\eta^s}\ar[d]^t &  N({M} \times I)\ar[r]^{\theta^s}\ar[d]^= &
L_{n + 1}^s  (M)\ar[d]^{l_0} \\ L_{n + 2}^h (M)
\ar[r]^{\gamma^h}\ar[d]^{\delta_L} & S^h({M} \times
I)\ar[r]^{\eta^h}\ar[d]^{\delta_S}&  N({M} \times I)\ar[r]^{\theta^h}&
L_{n + 1}^h  (M) \\ \hat H^n(\bbz_2;\wh(M)) \ar@{=}[r]& \hat
H^n(\bbz_2;\wh(M)) & &}
$$   

We want to understand the quotient group $\si(M)/\db(M)$, and the clue
is the following  observation:\smallskip

\begin{lemma}\label{SIquot}  $\si(M)/\db(M)= \tmop{im} \delta_S \subset   \hat
H^n(\bbz_2;\wh(M))\subset \wh(M)/\db(M)$.
\end{lemma}

\begin{proof} (See also \cite{QK}.) Recall  that an element of $S^h(M\times I)$ is represented by a homotopy equivalence
$f:W\to M\times I$ which is a homeomorphism on the boundary. Hence we can think of $W$ as an
$h$-cobordism from $M$, and as such it is clearly strongly inertial.  Since the map $\delta_S$ is
as induced by $(f:W\to M\times I)\mapsto \tau(W,M)$, the inclusion $\supseteq$ follows.\par

To prove the opposite inclusion, let $(W;M,N)$ be a strongly   inertial $h$-cobordism representing an element $z$ in
$\si(M)/\db(M)$, and let $H:N\times I\to M$ be a homotopy from the natural homotopy equivalence $h_W=r_M|N$ 
to a homeomorphism.  Define a map $W\to M$ as the
composite $W\xrightarrow{\approx} W\cup_N N\times I\to M$, where the last map is $H$ on the collar $N\times I$ and 
the retraction $r_M$ on $W$.  Combined with any map $(W;M.N)\to (I;0,1)$ this defines an element
of $S^h(M\times I)$ which image $z\in \si(M)/\db(M)$.
\end{proof}

We include the following corollary, which is our way of understanding 
Theorem 3.1 in \cite{QK} and its proof.

\begin{cor}\label{SIinv}
Let $f:M\to M'$ be a homotopy equivalence of closed manifolds.  Then
the induced isomorphism $f_*: \wh(M)\to \wh(M')$ restricts to
an isomorphism $f_*:SI(M)\to SI(M')$.
\end{cor}

\begin{proof}  We need to verify that $f_*(SI(M))\subseteq SI(M')$.\par
Lemma \ref{SIquot} and functoriality of the surgery exact sequence 
imply that the induced homomorphism $f_*: \wh(M)/D(M)\to \wh(M')/D(M')$
retricts to a homomorphism $f_*: SI(M)/D(M)\to SI(M')/D(M')$. In 
other words, if $x\in SI(M)$, then $f_*(x)=y+d$, where $y\in SI(M')$
and $d\in D(M')$. But then obviously also $f_*(x)\in SI(M')$.
\end{proof}

The most obvious way to try to prove Theorems \ref{DnotSI}
 and \ref{nonor} will now be to show that in these cases the homomorphism
$l_1$ in the diagram above is not onto. \par

In the case of Theorem \ref{DnotSI}, we need to study the map of {\em even} $L$-groups:   $l_1:L^s_{2m}(\pi)\to L^h_{2m}(\pi)$, 
where $\pi=\pi_1(M)$ and $2m=\tmop{dim}M+2$. Now assume that $\pi=\bbz_k$ is a cyclic group of {\em odd} order $k$.  
Then $l_1$ is injective. In fact, its image splits off as the free part plus a $\bbz_2$ (Arf invariant) if $m$ is odd. Hence, 
any other torsion in $L_{2m}^h(\bbz_k)$ maps nontrivially by $\delta_L$.\par
The extra torsion is computed from the Rothenberg sequence relating $L^h_*$ and $L^p_*$-groups:
 \[ \longrightarrow L_{2m+1}^p (\bbz_{k}) \longrightarrow H^2(\bbz_2;\widetilde{K_0} \bbz [\bbz_{k}]) \longrightarrow L_{2m}^h
   (\bbz_{k}) \longrightarrow L_{2m}^p (\bbz_{k}), \]
where the groups $L_{2m+1}^p (\bbz_{k})$ vanish by  \cite[Corollary 4.3, p.58]{BK}.\par
An example where $H^2(\bbz_2;\widetilde{K_0} \bbz [\bbz_{k}])$ is nontrivial is provided by \cite[Theorem 7.1, p.449]{KM},
where it is shown that $\widetilde {K_0}(\bbz[\bbz_{15})\approx \bbz_2$. Hence, if we choose $M$ to be any orientable, 
closed manifold of even dimension and fundamental group $\bbz_{15}$, then $\db(M)\ne \si(M)$.\par
To see that $\db(M)\ne 0$, recall that $\wh(\bbz_{15})\approx \bbz^4$ (see e.\,g. \cite[11.5]{Co}), and that the involution is 
trivial for abelian groups. Then
$\db(M)=2\wh(\bbz_{15})\approx \bbz^4$.\medskip

 For Theorem \ref{nonor}, consider the cyclic 2-group $Z_{2^k},\  k\geqslant 4$, with the nontrivial orientation character
 $\omega:\bbz_{2^k}\to \{\pm1\}$. Computations in \cite[Theorem 3.4.5]{Wa3} and 
 \cite[Theorem B and formula p.44]{CS2} give
 $$L^h_{2m+1}(\bbz_{2^k},\omega)\xrightarrow[\approx]{\delta_L} H^1(\bbz_2;Wh(\bbz_{2^k})^-)\approx (\bbz_{2^k})^{k-3},$$
 where the cohomology is with respect to the involution twisted by $\omega$.  
 
\medskip

The proof of Theorem \ref{periodic} goes by an argument similar to the proof 
of Theorem 1.3 in \cite{JK} (Theorem \ref{abelian} above).
 We need the following facts:\smallskip

\noindent{\em FACT 1}: The involution $- : \wh (\pi)
\rightarrow \wh (\pi)$ is trivial.

This is Claim 3 in [KS3] p.1527.
\smallskip

\noindent{\em FACT 2}: The homomorphism $l_1$ is surjective.

This is Claim 1 in [KS3] p.1527.
\smallskip 

\noindent{\em FACT 3}: The homomorphism $l_0$ is injective on the
image of $\theta^s$.

{\em{Proof of FACT 3}}. Since $\tmop{im} \theta^s \subseteq L_{n + 1}^s
(\pi_2)$, where $\pi_2$ is the Sylow 2-subgroup of $\pi (\tmop{cf} .
[\tmop{Wa} 2])$ it is enough to show that restriction $l_0 : L_{n + 1}^s
(\pi_2) \rightarrow L_{n + 1}^h (\pi_2)$ is injective. To this end note that
$\tmop{SK}_1 (\pi_2) = 0$ (cf. \cite{O}), where
\[ \tmop{SK}_1 (\pi_{}) \assign \tmop{Ker} (K_1 (\bbz [\pi])
   \longrightarrow K_1 (\mathbbm{Q} [\pi])) \]
Indeed $\pi_2$ is either generalized quaternionic or cyclic!

As a consequence $L_{n + 1}^s (\pi_2) \cong L_{n + 1}^{'} (\pi_2)$ where
$L_{\ast}^{'} (-)$ are the weakly simple $L$-groups of C.T.C. Wall from 
\cite{Wa3}.
Now, there is an exact sequence (cf. \cite[p. 78]{Wa3})
\[ 0 \to L_{2 n}^s (\pi_2) \longrightarrow L_{2 n}^h (\pi_2)
   \longrightarrow \wh' (\pi_2) \otimes \bbz_2 \longrightarrow
   L_{2 n - 1}^s (\pi_2) \longrightarrow L_{2 n - 1}^h (\pi_2) \to
   0 \]
and hence $l_{0|\tmop{im} \theta^s}$ is injective as claimed.

Given the Facts (1-3) the proof of Theorem \ref{periodic} is just a 
repetition of the
argument in \cite{JK}.

\

\section{Further remarks}

(1) Let $\pi$ be a finite group with $\tmop{SK}_1 (\pi) = 0$, for example any
dihedral group, or many nonabelian metacyclic groups, etc. (see \cite{O} for more
such groups). Then $\wh (\pi) \cong \wh' (\pi)$ is torsion free
and the involution $- : \wh (\pi) \rightarrow \wh (\pi)$ is
trivial (cf. \cite{O}). This is enough to extend Theorem 2 to this class of
fundamental groups.

Indeed, let $({W}^{n + 1} ; {M}^n, {N}^n)$ be a
strongly inertial $h$-cobordism, $n$ odd. We can assume $n \geqslant 5$. Let
$h : {N}^n \longrightarrow {M}^n$ be the natural homotopy
equivalence. Since $h$ is homotopic to a homeomorphism, then $\tau (h) = 0$. On
the other hand, $\tau (h) = 2 \tau ({W}^{n + 1}, {M}^n)$. This
implies $\tau ({W}^{n + 1}, {N}^n) = 0$ and ${W}^{n +
1} ={M}^n \times I$, i.e. $\si ({M}^n) = \{ 0 \}$.

\

(2) There are periodic groups $\pi$ with $\tmop{SK}_1 (\pi) \neq 0$. For
example groups containing $\bbz_p \times Q (8)$, where $p \geqslant 3$
is prime and $Q (8)$ is the quaternionic group of order 8.(cf. \cite{O}).

\

(3) There exist strongly inertial $h$-cobordisms with nontrivial 
Whitehead torsion
$\tau ({W}^{n + 1}, {M}^n, {N}^n)$ with $n$ odd $n
\geqslant 5$.

To be more specific, let $p$ be an odd prime and let $G$ be a $p$-group 
such that $SK_1(G)_{(p)}$ is non-trivial, for example the group given in 
Example 8.11 of \cite{O}, p.\,201. Then the argument on page 323 of \cite{O}
shows that 
that the involution $- :\wh (G) \rightarrow \wh (G)$ is nontrivial. Now let
${M}^n$, $n$ odd, $n \geqslant 5$ be a manifold with $\pi_1
({M}^n) \cong G$. Then the doubling construction gives a strongly inertial
$h$-cobordism $({W}^{n + 1} ; {M}^n, {M}^n)$ with
$\tau ({W}^{n + 1}; {M}^n)$ of the form $\tau_0 -\overline{\tau_0}$ for
$\tau_0 \in \wh (G)$. Choosing $\tau_0 \in
\wh (G)$ with $\tau_0 \neq \overline{\tau_0}$ gives the desired inertial
$h$-cobordism.

\

(4) Let $G$ be a finite group and ${M}^n, n \geqslant 5, n$ odd, a
closed manifold with $\pi_1 ({M}^n) \cong G$. The following is a
curious restatement of a special case of our problem.\smallskip

\noindent {\textbf Question}: {\tmem{Is $\si ({M}^n)
= \{ 0 \}$ if the involution $- : \wh (G) \rightarrow
\wh (G)$ is the identity?}} (`Only if' is trivial in this case, since 
$\{\tau-\bar \tau|\tau\in\wh (G)\}\subset \si({M}^n)$.)
\smallskip

\noindent {\textbf Comments}: (a) The answer is yes for $G$-finite abelian or
periodic.

(b) Suppose $\si ({M}^n) = \{ 0 \}$, and let $\tau_0 \in\wh (G)$ 
be given. Again the doubling construction
gives a strongly inertial $h$-cobordism $({W}^{n + 1}, {M}^n,{N}^n)$ 
with torsion $\tau = \tau_0 - \overline{\tau_0}$. Since
$\si ({M}^n) = \{ 0 \}$ then $\tau_0 = \overline{\tau_0}$, i.\,e.
the involution is trivial. On the other hand suppose the involution 
$- :\wh (G) \rightarrow \wh (G)$ is trivial and let 
$({W}^{n +1}, {M}^n, {N}^n)$ be a strongly inertial $h$-cobordism. Since
the natural homotopy equivalence $h : {M}^n \rightarrow {N}^n$
is homotopic to a homeomorphism, then 
$0 = \tau (h) = - 2 \tau ({W}^{n +1}, {M}^n)$. 
In particular, if $\wh (G)$ is torsion free, then the
involution $- : \wh (G) \rightarrow \wh (G)$ is  always trivial.
Hence, for all such groups  $\si ({M}^n) = \{ 0 \}$.

\

(5) There exist 4-dimensional inertial $s$-cobordisms which are not products!
(cf. \cite{CS}, \cite{KS2}).

\

\section{Addendum 1:  On topological invariance}\label{TopInv}

It is a consequence of the $s$-cobordism theorem and smoothing theory 
that if $M$ is a compact manifold and $\dim M\geq 5$,\
then the classification of $h$-cobordisms from $M$ up to isomorphism
relative to $M$ is the same in the three categories {\em TOP}, {\em PL} and 
{\em DIFF}.   For example, if $M$ is smooth and $(W,M)$ is a topological 
$h$-cobordism, then $W$ has a smooth structure,  unique up to 
concordance, extending that
of $M$, and if two such $h$-cobordisms are homeomorphic rel $M$, then 
they are also diffeomorphic rel $M$.

However, the following question is more subtle:
\begin{question}\label{ques}
Suppose $(W;M,N)$ is a smooth $h$-cobordism which is inertial in 
{\em TOP}, does it follow that it is also inertial in {\em DIFF}?
\end{question}

In other words: if $M$ and $N$ are homeomorphic, are they then also
diffeomorphic? (Similar questions can of course be asked for the 
pairs of categories ({\em DIFF,PL}) and ({\em PL,TOP }).) 

Note that this  indeed holds for the examples provided by the general 
results 
and constructions above; for example $D(M),\ A(M)$ and those obtained by 
connected sum with products of spheres, and in Lemma 8.1 of \cite{JK} we 
claimed that the answer is always yes.  However, this 
was based on a too optimistic application of the product structure
theorem for smoothings, and it does not hold as it stands\footnote{We
would like to thank Jean-Claude Hausmann for pointing out the error in
\cite{JK}.}.
We have,  unfortunately, not been  able to correct this in general, 
but here is a proof
in the case of {\em strongly inertial} $h$-cobordisms.

\begin{prop}\label{SITOP} Let $M$ be a smooth, compact manifold. If $W$
is a $\tmop{PL}$ $h$-cobordism from $M$, then $W$ has a smooth structure
compatible with the given structure on $M$, unique up to concordance. If
$W$ is strongly inertial in $\tmop{PL}$,  then it is also strongly
inertial in $\tmop{DIFF}$. \par

Replacing the pair of categories $(\tmop{DIFF,PL})$ by $(\tmop{DIFF,TOP})$ or
$(\tmop{PL,TOP})$, a similar result is true, provided $M$ has dimension at
least 5.
\end{prop}

\begin{proof} Denote by $\Gamma(M)$ the set of concordance classes of
smoothings of the underlying $PL$ manifold $M$.  By smoothing theory, 
this is a
homotopy functor. In particular, if $(W;M,N)$ is an $h$-cobordism, the
inclusions   $M\subset_{j_M} W$ and $N\subset_{j_N} W$ induce
restriction {\em  isomorphisms}
$$ \Gamma(M)\xleftarrow[\approx]{j_M^*} \Gamma(W) \xrightarrow[\approx]{j_N^*} \Gamma(N).$$ 	 
This proves the first part of the Lemma and also defines a unique
concordance class of structures on $N$. \par

Now let $M_\alpha$ be the given structure on $M$, $W_\alpha$ a structure
on $W$ restricting to $M_\alpha$ and $N_\alpha$ the restriction of this
again to $N$, such that $(W_\alpha;M_\alpha,N_\alpha)$ is a smooth
$h$-cobordism.  Observe that since $j_M$ has a homotopy inverse $r_M$,
the composite isomorphism $\Gamma(M)\to \Gamma(N)$ is induced by $
r_M{\scriptstyle\circ} j_N$, i.\,e.  the natural homotopy equivalence
$h_W$.  But if the $h$-cobordism is ($PL$) strongly inertial,
the isomorphism is also induced by a $PL$ homeomorphism $f$.  This means that
$N_\alpha$ is concordant to the smoothing $N_{f^*\!\alpha}$ on $N$
transported from $M_\alpha$ by $f$ in such a way that $f$ becomes 
a {\em diffeomorphism} between  $N_{f^*\!\alpha}$  and  $M_\alpha$.
\par

Let $(N\times I)_\beta$ be a concordance between $N_\alpha$ and
$N_{f^*\!\alpha}$, i.\,e. a smooth structure restricting to $N_\alpha$ on
$N\times\{0\}$ and  $N_{f^*\!\alpha}$ on $N\times\{1\}$. By the product
structure theorem (\cite[part I]{HM}) there
is a diffomorphism  $H:(N\times I)_\beta\to N_\alpha\times I$
restricting to the identity on $N\times \{0\}$.  Then $F(x,t)=H(f(x),t)$
defines a homotopy (in fact $PL$ isotopy) between $f$ and a
diffomorphism between $M_\alpha$ and $N_\alpha$. But $f$ was homotopic
to $h_W$.\par

The proofs in the other cases are analogous, but one now needs the
triangulation  theory of \cite{KSb}, which is only  valid in dimensions
$\geqslant 5$.  
\end{proof}

\begin{rem} If $\dim M=4$, Question \ref{ques} has a negative answer,
even in the strongly inertial case.  In fact, the first counterexamples
to the  $h$-cobordism theorem given by Donaldson in \cite{D} are even
{\em strongly} inertial, so even Proposition \ref{SITOP} 
(in case {\em (DIFF,TOP)}) fails in this
dimension.
\end{rem}

\medskip

\section{Addendum 2:  Comments on torsion}\label{comments}

{We collect here some useful observations concerning the Whitehead
torsions of homotopy equivalences of manifolds and relations with
$h$--cobordisms.}  \medskip

Recall that to a homotopy equivalence $f: K\to L$ of finite complexes is
associated a Whitehead torsion $\tau(f)=f_* \tau(M_f,K)\in Wh(L)$
\cite{Co}.  Then the torsion of an $h$-cobordism $(W,M)$ can be
expressed as
$$\tau(W,M)=r_*\tau(\iota) = - \tau(r),$$ 
where $\iota$ is the inclusion $M\subset W$ and $r$ is a retraction
$W\to M$. If $j:N\hookrightarrow W$ is the inclusion of the other end of
$W$, we can express the torsion of the natural homotopy  equivalence
$h=r\circ j$  as
\begin{eqnarray} \tau(h) = \tau(r)+r_*(\tau(j))&=&
-\tau(W;M)+r_*j_*(\tau(W;N))\notag \\
&=&-\tau(W;M)+(-1)^n\overline{\tau(W;M)}.\label{eq:tau-h}
\end{eqnarray}

The following observation shows that torsions of homotopy  equivalences
of manifolds can not be arbitrary, unlike for $h$-cobordisms.

\begin{lemma} Let $f:(N,\partial N)\to (M,\partial M)$ be a homotopy
equivalence between compact, oriented and connected manifolds of
dimension $n$, such that $f$ is a homeomorphism on the  boundary, and
let $\tau\in Wh(M)$ be its torsion. Then
$$\tau+(-1)^n\tau^*=0.$$
\end{lemma}

\begin{proof} There is a commutative diagram

$$\xymatrix{
C_*(N)\ar[r]^{f_\#}\ar[d] &C_*(M)\ar[d]\\ C_*(N,\partial
N)\ar[r]^{f_\#^{rel}} &C_*(M,\partial M)\\ C^*(N)\ar[u]^{D_N}&
C^*(M)\ar[l]^{f^\#}\ar[u]^{D_M}}
$$
of finitely generated $\bbz\pi_1(M)$--modules, where the lower vertical
maps are given by Poincar\'e duality.  (Everything with coefficients in
$\bbz\pi_1(M)$).  Then
$$ \tau(D_M)=\tau(f_\#^{rel})+f_*\tau(D_N)+f_*\tau(f^\#).$$
(Here $h_*$ is the map induces on Whitehead groups.)  The result now
follows, since  the Poincar\'e duality maps have vanishing  torsion,
$\tau(f_\#^{rel})=\tau(f_\#)=\tau$ and
$h_*(\tau(f^\#))=(-1)^n(\tau(f_\#))^*.$

\end{proof}

\noindent{\em Remark.}  More generally, without the assumption  that
$f|\partial N$ is a homeomorphism (or at least a simple homotopy
equivalence), we get the formula
$$\tau(f)-\tau(f|\partial M) + (-1)^n(\tau(f))^*=0.$$

\noindent{\em Example.}  Many finite groups have free Whitehead groups,
and  then it is known that the involution is trivial (Wall).  For an
even--dimensional closed manifold with one of these groups as
fundamental group, it follows that all homotopy equivalences are
simple.\medskip

The lemma is used in the following geometric proof of the Rothenberg
sequence for structure sets. We use the convention that $\sh(M)$
($\ss(M)$) denotes the structure set of maps which are homeomorphisms on
the boundary.

\begin{theorem} Let $M$ be a compact, oriented and connected manifold of
dimension $n$. Then there is an exact sequence of based sets ({\em
groups}, in the topological category)
$$\to \stop^h(M\times I)\xrightarrow{\theta}\widehat H^n(\bbz/2;Wh(M))\overset{\psi}{\to} \stop^s(M) \overset{\iota}{\to}
\stop^h(M)  \overset{\theta}{\to}  \widehat H^{n-1}(\bbz/2;Wh(M)).$$
\end{theorem}

\begin{proof} The map $\iota$ is the obvious forgetful map; $\psi$ and
$\theta$ will be defined below.\par We start with $\theta$. Recall that 
\begin{align*} \widehat H^{n-1}(\bbz/2;Wh(M))=& \,\{\tau\in
Wh(M)|\tau=(-1)^{n-1}\tau^*\}/\{\tau+(-1)^{n-1}\tau^*\}\\ =&\,\{\tau\in
Wh(M)|\tau+(-1)^{n}\tau^*=0\}/\{\tau-(-1)^{n}\tau^*\}.
\end{align*}

If $f:N\to M$ represents an element of $\stop^h(M)$, it then follows
from the lemma above that $\tau(f)$ represents an element of  $\widehat
H^{n-1}(\bbz/2;Wh(M))$. We have to show that this element is
well--defined.
\par Let $f':N'\to M$ represent the same element of $\stop^h(M)$ as
$f$. Then  there is an $h$--cobordism $W$ from $N$ to $N'$ and a map
$F:W\to M$  restricting to $f$ and $f'$ at the ends. 

\begin{equation}\label{eq:hrel}\xymatrix{ N\ar[d]_{\cap}\ar[dr]^f &\\
W\ar[r]^F & M\\ N'\ar[u]^{\cup}\ar[ur]_{f'} }
\end{equation}

Let $\sigma=\tau(W,N)$ be the torsion of the $h$--cobordism. By equation
(\ref{eq:tau-h} above we have $\tau(h)=-\sigma+(-1)^n\sigma^*$,  where
$h:N'\to N$ satisfies $f\circ h\simeq f'$. But then
$$\tau(f')=f_*\tau(h)+\tau(f),$$
and  $f_*\tau(h)$ is trivial in $\hat H^{n-1}(\bbz/2;Wh(M))$.\smallskip

Trivially $\theta\circ\iota=0$. Suppose now that $f\in\stop^h(M)$
satisfies  $\theta(f)=0$, i.\,e.   $\tau(f)=\sigma-(-1)^n\sigma^*$, for
some  $\sigma\in Wh(M)$.  Let $W$ be an $h$--cobordism with one end
equal to $N$ and with Whitehead torsion
$\tau(W,N)=f_*^{-1}(\sigma)$. Extension of  the map $f$ yields a diagram
like (\ref{eq:hrel}). Then $f'$ and $f$ represent the same element of
$\sh(M)$, and $\tau(f')=0$.\medskip

The last construction is also used to define $\psi$\,: let  $\tau\in
Wh(M)$ represent an element of $\hat H^{n}(\bbz/2;Wh(M))$, i.\,e.
$\tau=(-1)^n\tau^*$, and  let $W$ be an $h$--cobordism with one end
equal to $M$ and torsion $\tau$.  If the other end is $N$, the natural
homotopy equivalence $h:N\to M$  has torsion
$\tau(h)=-\tau+(-1)^n\tau^*=0$ and we set $\psi(\tau)=h$.\par If
$\tau=\sigma+(-1)^n\sigma^*$ we can choose $W$ to be the ``double'' of
an $h$--cobordism with torsion $\sigma$ (Milnor, Lemma 11.4), and the
construction gives $h=\text{id}_M$. Hence $\psi$ is well--defined.\par
The construction of $\psi$ is illustrated by the following special case
of  diagram (\ref{eq:hrel}): 
\begin{equation}\label{eq:psidef}\xymatrix{
M\ar[d]_{\cap}\ar[dr]^{\text{id}_M} &\\ W\ar[r]^F & M\\
N\ar[u]^{\cup}\ar[ur]_{h} }
\end{equation}

Exactness at $\ss(M)$ follows when we observe that this diagram also
expresses precisely that $h$ is equivalent to $\text{id}_M$ in
$\sh(M)$. 
\end{proof}

\noindent{\em Remarks.}  (1) The sequence can be continued to a long
exact sequence of groups to the left by  hooking it up in the obvious
way with the sequences for $M\times I$, $M\times I^2$ and so
on.\smallskip 

(2) The maps $L^*_{n+1}(M)\to \mathcal S^*(M)$ in the surgery sequences
($*=$s,\,h) give an obvious map from the $L$-theory Rothenberg sequence,
and  an interlocking braid (continuing to the left)

\begin{eqnarray} \xymatrix@C-3pc@R-1pc{ N(M\times I)\ar[rd]
\ar@/^+1pc/[rr]  && L^h_{n+1}(M) \ar[rd] \ar@/^1pc/[rr]  && \widehat
H^{n-1}(\bbz/2;Wh(M)) \ar[rd]  \label{eq:rotharray} & \\ &
{L^s_{n+1}(M)} \ar[ur] \ar[dr]  && {\sh(M)} \ar[ur] \ar[dr] &&
L^s_{n}(M)\ar[dr] & \\ \widehat H^n(\bbz/2;Wh(M)) \ar[ur]
\ar@/_1pc/[rr]^J  && \ss(M) \ar[ur] \ar@/_1pc/[rr]^{\sigma_*}  && N(M)
\ar[ur]\ar@/_1pc/[rr] && L^h_{n}(M) \\ }
\end{eqnarray}


Bj\o rn Jahren

Department of Mathematics

University of Oslo

0316 Oslo

Norway

bjoernj@math.uio.no \\

Slawomir Kwasik

Department of Mathematics

Tulane University

New Orleans, LA 70122

USA

kwasik@tulane.edu

\end{document}